\documentstyle[amssymb,amsfonts,12pt]{amsart}
\newtheorem{theorem}{Theorem}[section]
\newtheorem{lemma}[theorem]{Lemma}
\newtheorem{corollary}[theorem]{Corollary}
\newtheorem{proposition}[theorem]{Proposition}
\newtheorem{definition}[theorem]{Definition}

\theoremstyle{remark}
\newtheorem{remark}[theorem]{Remark}

\def\QSet{\mbox{\rm\kern.24em
\vrule width.03em height1.48ex depth-.051ex \kern-.26em Q}}

\def\D{{\mathcal D}}
\def\T{{\mathbb T}}
\def\P{{\bf P}}\def\p{{\bf p}}
\def\R{{\mathbb R}}

\def\N{{\mathbb N}}
\def\C{{\mathbb C}}
\def\Q{{\bf Q}}
\def\Z{{\mathbb Z}}
\def\I{{\mathcal I}}
\def\F{{\mathcal F}}
\def\J{{\bf J}}

\def\\xi{{\bf \xi}}

\def\be#1{\begin{equation}\label{#1}}

\def\size{{\operatorname{size}}}\def\mass{{\operatorname{mass}}}

\def\O{{\operatorname{O}}}

\def\bas{\begin{align*}}
\def\eas{\end{align*}}
\def\bi{\begin{itemize}}
\def\ei{\end{itemize}}
\newenvironment{proof}{\noindent {\bf Proof} }{\endprf\par}
\def \endprf{\hfill  {\vrule height6pt width6pt depth0pt}\medskip}
\def\emph#1{{\it #1}}

\begin{document}
\title{A guide to Carleson's Theorem}

\author{Ciprian Demeter}
\address{Department of Mathematics, Indiana Unversity, Bloomington, IN 47405}
\email{demeterc@@indiana.edu}
\thanks{}

\subjclass[2010]{Primary: 42A20, Secondary: 42A45}

\keywords{Fourier series, singular integrals, Carleson's Theorem}

\date{}

\dedicatory{}

\begin{abstract}
This paper is meant to be a gentle introduction to Carleson's Theorem on pointwise convergence of Fourier series.
\end{abstract}

\maketitle
\section{introduction}

Let $$S_nf(x)=\sum_{k=-n}^n\widehat{f}(k)e^{2\pi ikx},$$
be the partial Fourier series of the $L^1$ function $f$ on $[0,1]$. In 1966, Lennart Carleson has proved the following long standing conjecture.
\begin{theorem}[\cite{Car}]
\label{Car}
For each $f\in L^2([0,1])$, the Fourier series $S_nf$ converge almost everywhere to $f$.
\end{theorem}
Soon after that, a slight modification of Carleson's method allowed Hunt \cite{Hun} to extend the result to $L^p$ functions, for $p>1$.

Theorem \ref{Car} has since received many  proofs, most notably by Fefferman \cite{Fe} and by Lacey and Thiele \cite{LT-C}.
The impact of Carleson's Theorem has increased in recent years thanks to its connections with Scattering Theory \cite{MTT1}, Egodic Theory \cite{DLTT}, \cite{DT},  the theory of directional singular integrals in the plane \cite{LaLi1}, \cite{LaLi2}, \cite{Dem1}, \cite{DeDi}, \cite{Ba}, \cite{BT} and the theory of operators with quadratic modulations \cite{Lie1}, \cite{Lie2}. A more detailed description  can be found in \cite{Lac}. These connections have motivated the discovery of various new arguments for Theorem \ref{Car}. While these arguments share some similarities, each of them has a distinct personality. Along these lines, it is interesting to note that for almost every specific application of Carleson's Theorem in the aforementioned fields, only one of the arguments will do the job.

All the arguments for Theorem \ref{Car} are technical. To present the main ideas in a transparent  way, we will instead analyze the closely related Walsh-Fourier series, which we recall below.

For $n\ge 0$  the $n-$th Walsh function $w_n$ is defined recursively by the formula
$$w_0=1_{[0,1)}$$
$$w_{2n}=w_n(2x)+w_n(2x-1)$$
$$w_{2n+1}=w_n(2x)-w_n(2x-1).$$

Given $f:[0,1]\to\C$ we recall the partial Walsh-Fourier series of $f$
$$S_n^{W}f(x)=\sum_{k=0}^n\langle f, w_k\rangle w_k(x).$$

The following theorem was proved by Billard, by adapting Carleson's methods.
\begin{theorem}[\cite{Bil}]
\label{thm:Bil}

For each $1<p\le \infty$ and each $f\in L^p([0,1])$, the series $S_n^{W}f(x)$ converges almost everywhere to $f(x)$.
\end{theorem}

We present a few proofs of Theorem \ref{thm:Bil} which are  translations of their Fourier analogues. In each case the translation can be done in more than one way, the proofs presented reflect author's taste.
While a few of the features of the original proofs from the Fourier case will be lost in translation, the main line of thought will be preserved essentially intact in the Walsh case. Very little originality is claimed by the author.

This paper is by no means a complete guide to Carleson's Theorem, in particular we shall make no attempt to describe in detail any of its afore mentioned applications. The main goal is to give a self contained but concise survey of some of the main arguments in the literature.

\subsection*{Acknowledgements}The author would like to thank  Christoph Thiele for clarifying discussions on the argument from \cite{Lie2} and to Fangye Shi for carefully reading the original manuscript and pointing out a few typos. Part of the material in this paper was organized while teaching a class on Harmonic Analysis. Many thanks to Ben Krause for a careful reading and for pointing out a few inaccuracies. The author is grateful to his students Francesco Di Plinio and Prabath Silva for motivating him and to his collaborators for enriching his understanding of time-frequency analysis.

\section{The Walsh phase plane}

It turns out that there is a multiscale description for $S_n^{W}f$. Let $\D_+$ denote the collection of all dyadic intervals which are subsets of $\R_+=[0,\infty)$. We call $\R_+\times \R_+$ the Walsh phase plane.

\begin{definition}
A tile $p=I_p\times \omega_p$ is a rectangle of area one, such that $I_p,\omega_p\in\D_+$. A bitile $P=I_P\times \omega_P$ is a rectangle of area two, such that $I_P,\omega_P\in\D_+$. Let $\omega_{P_l}, \omega_{P_u}$ be the left (or lower) and right (or upper) halves of $\omega_P$. We will denote by $P_l=I_P\times \omega_{P_l}$ and $P_u=I_P\times \omega_{P_u}$ the lower and upper tiles of $P$. We denote by $\P_{all}$ the collection of all bitiles.
\end{definition}

Given a tile $p=[2^jm, 2^{j}(m+1)]\times [2^{-j}n, 2^{-j}(n+1)]$ we define the associated Walsh wave packet
$$W_p(x)=2^{-j/2}w_n(2^{-j}x-m).$$

To understand the relevance of the Walsh phase plane, we recall a few tools from \cite{Thi1}, see also \cite{Ter1}.

Every $x\in \R_{+}=[0,\infty)$ can be identified uniquely with a doubly-infinite set of binary digits $a_n=a_n(x)$ such that
$$x=\sum_{n\in\Z}a_n2^n$$
where $a_n\in\{0,1\}$ and $\liminf_{n\to-\infty}a_n=0$. Note that $a_n$ is eventually zero as $n\to\infty$. We define two operations on $\R_{+}$. First, $$x\oplus y:=\sum_{n}b_n2^n$$
where
$$b_n:=a_n(x)+a_n(y)\mod 2.$$
We caution the reader that $b_n$ is not always the same as $a_n(x\oplus y)$, as the example $x=\sum_{n<0:\,n\text{ odd}}2^n$, $y=\sum_{n<0:\,n\text{ even}}2^n$ shows. Also, since $1=x\oplus y=x\oplus z$
where $z=1+\sum_{n<0:\,n\text{ odd}}2^n$, $(\R_+,\oplus)$ is not, strictly speaking a group. We will be content with observing that for all practical purposes  $(\R_+,\oplus)$ can be thought of as being a group, in the sense that $\oplus$ behaves like a genuine group operation if we exclude pairs $x,y$ of zero product Lebesgue measure.

Define the second operation by $$x\otimes y:=\sum_{n}c_n2^n$$
where
$$c_n:=\sum_{m\in\Z}a_m(x)a_{n-m}(y)\mod 2.$$
We note that this  sum is always finite. If we neglect zero measure sets, $(\R_+,\otimes,\oplus)$ can be thought of as being  a field with characteristic two. It will be implicitly assumed that various equalities to follow hold outside zero measure sets.

Define the function $e_W:\R_{+}\to\{-1,1\}$ such that $e_W(x)=1$ when $a_{-1}(x)=0$ and $e_W(x)=-1$ when $a_{-1}(x)=1$. This  1-periodic function is the Walsh analogue of $e^{2\pi ix}$. It is easy to check that
$$w_n(x)=e_W(x\otimes n)1_{[0,1]}(x),$$
thus $w_n$ can be thought of as being the Walsh analogue of $e^{2\pi i nx}$.
Also, for each tile $p=I_p\times \omega_p$ we have
$$W_p(x)=\frac{1}{|I_p|^{1/2}}w_0(\frac{x-l(I_p)}{|I_p|})e_W(x\otimes l(\omega_p)),$$
where $l(J)$ denotes the left endpoint of $J$. A simple computation shows that for each bitile $P$
\begin{equation}
\label{e54757897985960869078089=-}
W_{P_l}(x)=W_{P_u}(x)e_{I_P}(x)
\end{equation}
where $e_I(x)=1$ on $I_l$ and $e_I(x)=-1$ on $I_r$.

The collection of all wave packets $W_p$ where $p$ ranges over all tiles with fixed scale forms a complete orthonormal system in $L^2(\R_+)$ and thus
$$f=\sum_{p:\;|I_p|=2^j}\langle f,W_p\rangle W_p.$$

We introduce the Walsh (also called Walsh-Fourier) transform of a function $f:\R_+\to\C$ to be
$$\F_Wf(\xi)=\widehat{f}(\xi):=\int e(x\otimes\xi)f(x)dx.$$
It is easy to see that its inverse $\F_W^{-1}$ coincides with $\F_W$.

Arguably the most important feature that makes the Walsh phase plane technically simpler than its Fourier counterpart  is the absence of the strong form of the  "Uncertainty Principle". This allows the existence of functions compactly supported in both time and frequency.
The best example is $1_{[0,1]}$, which  equals its Walsh transform. A quick computation shows that for each interval $I\in\D_+$
\begin{equation}
\label{e12}
\widehat{1_{I}}(\xi)=|I|1_{[0,|I|^{-1}]}(\xi)e(\xi\otimes x_I)
\end{equation}
where $x_I$ is an arbitrary element of $I$.
Similarly
$$\widehat{W_p}(\xi)=\frac{1}{|\omega_p|^{1/2}}w_0(\frac{\xi-l(\omega_p)}{|\omega_p|})e_W(\xi\otimes l(I_p)).$$
 Thus $W_p$ is spatially supported in $I_p$ while its Walsh transform is supported in $\omega_p$. An application of Plancherel's theorem  shows that
\begin{equation}
\label{e11}
\langle W_p,W_{p'}\rangle=0
\end{equation}
whenever the tiles $p$ and $p'$ do not intersect.

The following partial relation of order was introduced by C. Fefferman \cite{Fe}.
\begin{definition}[Order]
For two tiles or bitiles $P,P'$ we write $P\le P'$ if $I_P\subset I_{P'}$ and $\omega_{P'}\subset \omega_{P}$
\end{definition}
Note that $P$ and $P'$ are comparable under $\le$ if and only if they intersect as sets. We will refer to maximal (or minimal) tiles (or bitiles) with respect to $\le$ as simply being maximal (or minimal).

\begin{definition}[Convexity]
A collection $\P$ of bitiles is called convex if whenever  $P,P''\in\P$, $P'\in\P_{all}$ and $P\le P'\le P''$, we must also have $P'\in\P$.
\end{definition}

For a collection $\p$ of tiles or bitiles we denote by $A(\p)=\bigcup_{p\in\p}I_p\times \omega_p$ the region in $\R_+^2$ covered by them.

We will use the fact (see Lemma 2.5 in \cite{Thi1}) that for each convex set of bitiles $\P$, the region $A(\P)$  can be written (not necessarily in a unique way) as a disjoint union of tiles $\p$
\begin{equation}
\label{e10}
A(\P)=A(\p)
\end{equation}

We can identify any region in the Walsh phase plane which is a finite union of pairwise disjoint tiles $p\in\p$ with the subspace of $L^2(\R_+)$ spanned by $(W_p)_{p\in\p}$. Indeed, it turns out that if two such collections  $\p$ and $\p'$ of tiles cover the same area in the phase plane, then $(W_p)_{p\in\p}$ and $(W_p)_{p\in\p'}$ span the same vector space in $L^2(\R_+)$, see Corollary 2.7 in \cite{Thi1}). In particular
\begin{equation}
\label{33}
\sum_{p\in\p}\langle f,W_p\rangle W_p(x)=\sum_{p\in\p'}\langle f,W_{p}\rangle W_p(x).
\end{equation}
The projection operator onto this subspace
$$\Pi_{\p}f(x)=\sum_{p\in\p}\langle f,W_p\rangle W_p(x)$$
will be referred to as the {\em phase space projection} onto $\p$. \eqref{e11} guarantees that $(W_p)_{p\in\p}$ forms an orthonormal basis of the range of $\Pi_\p$. In particular, if $\P$ is a convex union of bitiles and $\p$ satisfies \eqref{e10}, then we abuse notation and define
$$\Pi_{\P}f:=\Pi_{\p}f. $$

An easy induction argument proves

\begin{lemma}[Corollary 2.4, \cite{Thi1}]
\label{2diffways}
Let $\p,\p'$ be a finite collections of pairwise disjoint tiles such that $A(\p')\subset A(\p)$. Then there exists a collection $\p''$ of pairwise disjoint tiles which includes all the tiles in $\p'$ such that $A(\p)=A(\p'')$.
\end{lemma}

Lemma \ref{2diffways} and \eqref{e11} will imply that for each convex collection $\P$ of bitiles and for each  $P\in \P$
\begin{equation}
\label{e16}
\langle f,W_{p}\rangle=\langle \Pi_\P f,W_{p}\rangle,\;p\in\{P_u,P_l\}.
\end{equation}

\begin{remark}
\label{rem1}
We mention that the relation of order as well as concepts such as convexity and  phase space projections can be extended naturally to the two (or higher) dimensional case. This will be explored in Section \ref{nochoice}.
\end{remark}

Fix now $n\ge 0$. Note that $S_n^{W}f(x)=\Pi_{\p_n}f(x)$, where $\p_n$ is the collection of the tiles $[0,1]\times [k,k+1]$, $0\le k\le n$. We will partition $A(\p_n)=[0,1]\times [0,n+1]$ in a different way. Namely, for each point $(x,\xi)\in [0,1]\times [0,n+1]$, there exists a unique bitile $P$ such that $(x,\xi)\in P_l$ and $(x,n+1)\in P_u$. This bitile is precisely the minimal one such that $n+1,\xi\in\omega_P$ and $x\in I_P$. Note that the tiles ${P_l}$ corresponding to all these $P$ will partition $[0,1]\times [0,n+1]$. But then \eqref{33} will imply that

$$S_n^{W}f(x)=\sum_{P\in\P_{all}:n+1\in \omega_{P_u}}\langle f,W_{P_l}\rangle W_{P_l}(x),$$
in particular
$$\sup_{n\ge 0}|S_n^{W}f(x)|=|\sum_{P\in\P_{all}}\langle f,W_{P_l}\rangle W_{P_l}(x)1_{\omega_{P_u}}(N(x))|$$
for a suitable function $N:\R_+\to\N$. The roles of $P_u$ and $P_l$ can be interchanged, without altering the nature or the difficulty.  For pedagogical reasons  we choose to work with
the model sums
$$C_{\P}f(x)=\sum_{P\in\P}\langle f,W_{P_u}\rangle W_{P_u}(x)1_{\omega_{P_l}}(N(x)),$$
where $\P\subset \P_{all}$ and $f:\R_{+}\to\C$. Using the standard approximation argument combined with the almost everywhere convergence of $S_n^Wf(x)$ for characteristic functions of intervals, Theorem \ref{thm:Bil} will follow from the following inequality.
\begin{theorem}
\label{thm:main}
We have for each $1<p<\infty$, $N:\R_+\to\R_+$ and $f\in L^p(\R_+)$
$$\|C_{\P_{all}}f\|_p\le C_p \|f\|_p.$$
The constant $C_p$ does not depend on $f$ and $N$.
\end{theorem}

Note that to recover Theorem \ref{thm:Bil} we could restrict attention to functions $f$ on $[0,1]$ and to the bitiles spatially supported in $[0,1]$. We will do so in some, but not all the proofs to follow. We will always allow the choice function $N$ to take any value in $\R_+$, not just integers.

\section{Estimates for a single tree}

Throughout the paper we will denote by $Mf$ the Hardy-Littlewood maximal function of $f$. The various implicit constants hidden in the notation $\lesssim$ will typically be allowed to depend on the H\"older exponents $p,p_i,s_i$ etc.

Fefferman \cite {Fe} organized the bitiles in structures that he named {\em trees} and {\em forests}. The restriction $C_T$ to a tree will be a  Calder\'on-Zygmund object which can be investigated with classical methods. The contribution of forests is controlled by using various forms of orthogonality between the tree operators $C_T$.
All the  approaches described in the following sections will rely on this strategy.

\begin{definition}
Let $I_T\in \D_+$ and $\xi_T\in \R_+\setminus\{n2^{-k}:n,k\in\Z\}$. A tree $T$ with top data $(I_T,\xi_T)$  is a collection of bitiles such that $I_P\subset I_T$ and $\xi_T\in \omega_P$ for each $P\in T$. If
 $P_T\in \P_{all}$ is such that
$P\le P_T$ for each $P\in T$, we call $P_T$ a top bitile for $T$. Note that such a $P_T$ is not unique and $T$ need not contain a top bitile.

A tree is called overlapping if the tiles $\{P_u:P\in T\}$ intersect. A tree is called lacunary if the tiles $\{P_l:P\in T\}$ intersect.
\end{definition}
Each tree can be decomposed as $T=T_l\cup T_o$ where
$$T_l=\{P\in T:\xi_T \in \omega_{P_l}\}$$
$$T_o=\{P\in T:\xi_T\subset \omega_{P_u}\}$$
Note that $T_l$ is lacunary while $T_o$ is overlapping. Moreover, if $T$ is convex then so are the trees   $T_l$ and $T_o$.  These observations  will allow us to always assume the tree we deal with  is either overlapping or lacunary.

The classical example of lacunary tree is the {\em Littlewood-Paley} tree consisting of all bitiles of the form $P^I:=I\times [0,|I|^{-1}]$, $I\in\D_+$.
Note that for each $P^I$ in the Littlewood-Paley tree we have
$W_{P^I_u}(x)=h_I(x)$, where $h_I$ is the $L^2$ normalized Haar function equal to $|I|^{-1/2}$ on the left half $I_l$ and to $-|I|^{-1/2}$ on the right half $I_u$. Recall that if $\I$ is a subset of the Littlewood-Paley tree we have
\begin{equation}
\label{e14}
\|\sum_{P^I\in\I}\langle f, h_I\rangle h_I\|_p\lesssim \|f\|_p,\;\;\;\;1<p<\infty.
\end{equation}

For a lacunary tree $T$ we denote
$$\O_Tf(x)=\sum_{P\in T}\langle f,W_{P_u}\rangle W_{P_u}(x).$$
We have the following generalization of \eqref{e14}.

\begin{lemma}[Single tree estimate: singular integral]
\label{lemasingtree1}
Let $T$ be a lacunary tree. Then for each $1<p<\infty$
$$\|\O_Tf\|_p\lesssim \|f\|_p.$$
\end{lemma}
\begin{proof}

Recall that $\widehat{W_{P_u}}$ is supported in  $\omega_{P_u}$. Call $\Omega$ the collection of all intervals $\omega_{P_u},\; P\in T$ and note that they are pairwise disjoint and sit within distance smaller than their length from $\xi_T$. By the Walsh version of the Littlewood-Paley Theorem applied to $\O_Tf$ and then to $f$ we have
$$\|O_Tf\|_p\lesssim \|(\sum_{\omega\in\Omega}|\F_W^{-1}(1_\omega\F_W(\O_Tf))|^2)^{1/2}\|_p=$$
$$=\|(\sum_{\omega\in\Omega}|\sum_{P\in T:\,|\omega_{P_u}|=|\omega|}\langle f,W_{P_u}\rangle W_{P_u}|^2)^{1/2}\|_p=\|(\sum_{P\in T}\frac{|\langle f,W_{P_u}\rangle|^2}{|I_P|}1_{I_P})^{1/2}\|_p\lesssim \|f\|_p.$$
\end{proof}

The next lemma shows that the operator $C_T$ restricted to an  tree $T$ is a maximal function, if $T$ is overlapping and a maximal truncation of a discrete  Calder\'on-Zygmund operator, if $T$ is lacunary.

\begin{lemma}[Single tree estimate: maximal truncations]
\label{furfy58t78t585igutgutih}
Let $T$ be a tree. Then for each $1<p<\infty$
$$\|\sum_{P\in T}\langle f,W_{P_u}\rangle W_{P_u}(x)1_{\omega_{P_l}}(N(x))\|_{L^p(\R_+)}\lesssim \|f\|_p.$$
\end{lemma}
\begin{proof}
It suffices to prove the lemma when $T$ is either lacunary or overlapping.
We start with the lacunary case.
Note that for each $x$ there is $k=k(x)$ such that
$$\sum_{P\in T}\langle f,W_{P_u}\rangle W_{P_u}(x)1_{\omega_{P_l}}(N(x))=\sum_{P\in T\atop_{|\omega_P|\ge 2^{k}}}\langle f,W_{P_u}\rangle W_{P_u}(x).$$
Note that if $P,P'\in T$ and $|\omega_P|>|\omega_{P'}|$ then $\omega_{P_u}$ is disjoint from and sits at the right of $\omega_{P_u'}$. Thus there will exist an interval $\omega=\omega(x)$ which contains all $\omega_{P_u}$ with $|\omega_P|\ge 2^{k}$ and which will have empty intersection with all $\omega_{P_u}$ satisfying $|\omega_P|< 2^{k}$. We can thus write
$$\sum_{P\in T\atop_{|\omega_P|\ge 2^{k}}}\langle f,W_{P_u}\rangle W_{P_u}(x)=\F_W^{-1}[1_\omega\F_W[\O_Tf]](x).$$
Using \eqref{e12} and the fact that $x\oplus [0,|\omega|^{-1}]$ is an interval of length $|\omega|^{-1}$ containing $x$ we get
$$|\sum_{P\in T\atop_{|\omega_P|\ge 2^{k}}}\langle f,W_{P_u}\rangle W_{P_u}(x)|\le |\omega|\int_{x\oplus [0,|\omega|^{-1}]}|\O_Tf(y)|dy\lesssim M(\O_Tf)(x).$$
It now suffices to apply Lemma \ref{lemasingtree1}.

Assume next that $T$ is overlapping. This case is immediate by noting that  for each $x$
$$|\sum_{P\in T}\langle f,W_{P_u}\rangle W_{P_u}(x)1_{\omega_{P_l}}(N(x))|=|\langle f,W_{P_u}\rangle W_{P_u}(x)|\lesssim M(f)(x),$$
where $P$ is the unique (possibly nonexisting) bitile with $(x,N(x))\in P_l$.
\end{proof}

We remark that \eqref{e16} implies that whenever $T$ is convex and $P\in T$ we have $\langle f,W_{P_u}\rangle=\langle \Pi_Tf,W_{P_u}\rangle$. Thus the result of Lemma \ref{furfy58t78t585igutgutih} can also be written  in a localized form
\begin{equation}
\label{e17}
\|C_Tf\|_p\lesssim \|\Pi_Tf\|_p
\end{equation}

We close this section with proving $L^p$ estimates for the phase space projection associated with a tree
\begin{proposition}
\label{hfuyhrufywiei3u884u8}
For each convex tree $T$ and each $1< p\le \infty$
$$\|\Pi_Tf\|_{p}\lesssim\|f\|_p.$$
 \end{proposition}
\begin{proof}
We first observe that for each $x\in I_T$, $\Pi_Tf(x)=\Pi_Pf(x)$ where $P$ is the minimal bitile in $T$ with $x\in I_P$. Indeed, according to Lemma \ref{2diffways}, there exists a collection $\p''$ of pairwise disjoint tiles which includes $P_u$ and $P_l$ such that $$\Pi_Tf(x)=\Pi_{P}f(x)+\sum_{p\in \p''\setminus\{P_u,P_l\}}\Pi_{p}f(x).$$
But $P$ is minimal, hence $\{x\}\cap I_p\cap I_P=\emptyset$ for each  $p\in \p''\setminus\{P_u,P_l\}$.

Finally, note that for each bitile $P$ and each $x\in I_P$ we have $$|\Pi_Pf(x)|\le |I_P|^{-1/2}(|\langle f,W_{P_u}\rangle|+|\langle f,W_{P_l}\rangle|)\le \frac{2}{|I_P|}\int_{I_P}|f|\lesssim Mf(x).$$
\end{proof}

\section{Size and pointwise estimates outside exceptional sets}
\label{sec:3}
We are now ready to see the first proof of Theorem \ref{thm:main}. This argument bears some resemblance to the original argument of Carleson \cite{Car}. A form of this argument has been used in \cite{MTT1} in the Walsh case, while the proof of the Fourier case is hidden in \cite{DLTT}. This type of argument proved instrumental in applications to the Return Times Theorem \cite{DLTT} and the directional Hilbert transform in the plane \cite{DeDi}.

The main tool is the {\em size} of a collection of bitiles, a concept introduced by Lacey and Thiele \cite{BHT} in the Fourier case and by Thiele\footnote{it is there referred to as {\em density} } \cite{Thi1} in the Walsh case.

\begin{definition}
The size of a collection $\P$ of bitiles with respect to a function $f:\R_+\to \C$ is defined as
$$\size_f(\P)=\sup_{P\in\P}\frac{\|\Pi_Pf\|_2}{|I_P|^{1/2}}.$$
\end{definition}
The next two propositions record some of the key features of $\size_f(\P)$.
\begin{proposition}
\label{dkjuyer8fu4589t75899rio3t}
Let $T$ be a convex tree. Then for each $1\le p\le \infty$
\begin{equation}
\label{e18}
\|\Pi_Tf\|_p\lesssim \size_f(T)|I_T|^{1/p}
\end{equation}
\end{proposition}
\begin{proof}
It suffices to prove that $\|\Pi_Tf\|_{\infty}\lesssim \size_f(T)$. The proof of Proposition \ref{hfuyhrufywiei3u884u8} shows that for each $x\in I_T$, $\Pi_Tf(x)=\Pi_Pf(x)$ where $P$ is the minimal bitile in $T$ with $x\in I_P$.
To close the argument, observe that for each bitile $P$ we have $$\|\Pi_Pf\|_{\infty}\le |I_P|^{-1/2}(|\langle f,W_{P_u}\rangle|+|\langle f,W_{P_l}\rangle|)$$$$\le |I_P|^{-1/2}\sqrt{2}(|\langle f,W_{P_u}\rangle|^2+|\langle f,W_{P_l}\rangle|^2)^{1/2}=|I_P|^{-1/2}\sqrt{2}\|\Pi_Pf\|_{2}$$

\end{proof}
\begin{proposition}\label{djhfdftyert}For each $\P$ and $f:\R_+\to\C$,
$$\size_f(\P)\lesssim \sup_{P\in\P}\inf_{x\in I_P}M(f)(x)$$
\end{proposition}
\begin{proof}
This is immediate, since for each $P$ $$\|\Pi_Pf\|_2=(|\langle f,W_{P_u}\rangle|^2+|\langle f,W_{P_l}\rangle|^2)^{1/2}\le \frac{\sqrt{2}}{|I_P|^{1/2}}\int_{I_P}|f|.$$
\end{proof}

\begin{definition}A forest $\F$ is a finite collection of bitiles which consists of a disjoint union of convex trees.
\end{definition}
We note that each finite convex collection $\P$ of bitiles can be turned into a forest, possibly in more than one way. Indeed, start with a maximal element $P$ from $\P$, and construct the maximal tree in $\P$ with top bitile $P$. Remove this tree $T$ from $\P$ and note that $\P\setminus T$ remains convex. Repeat the procedure with $\P\setminus T$ replacing $\P$, to select the next tree. Iterate this until all the bitiles from the original $\P$ are selected.

We will sometimes use the notation $N_\F$ for the counting function of a forest $\F$
$$N_\F(x)=\sum_{T\in\F}1_{I_T}(x).$$

A key idea in many  of the approaches to Carleson's Theorem is to split the bitiles into forests with a certain size.

\begin{lemma}

 \label{sizelemma}
Let $\P$ be a finite convex collection of bitiles and let $f:\R_+\to\C$. Then
$
\P= \P_{hi} \cup \P_{{lo}}
,$
such that \begin {itemize}
\item[$\cdot$]both $\P_{{lo}}$ and $ \P_{hi} $ are convex
\item[$\cdot$]$\size_f(\P_{{lo}} )\leq \frac12\size_f(\P)$:
\item[$\cdot$] $\P_{hi}$ is a convex forest with trees $T\in\F$ satisfying
$$
\sum_{T\in\F}  |I_T|\lesssim \size_f(\P)^{-2}\|f\|^2_2.
$$
\end{itemize}
\end{lemma}
\begin{proof} This is a recursive procedure. Set $\P_{{stock}}:=\P$ and $\F=\emptyset$. Select a maximal  bitile  $t \in \P_{{stock}}$ such that
$$
\|\Pi_t f\|_2|I_t|^{-1/2} > \frac{\size_f(\P)}{2}
$$
Define
$$
T(t) = \{P \in \P_{{stock}}: P \le t\}.
 $$
and note that since $\P_{{stock}}$ is convex, both $\P_{{stock}}\setminus T(t)$ and the tree $T(t)$ will be convex. Add $T(t)$ to the family $\F$. Reset  $\P_{{stock}}:=\P_{{stock}}\setminus T(t)$, and restart the procedure.

The algorithm is over when there is no $t$ to be selected. Then define $ \P_{hi} $ to consist of the union of  all bitiles in all the trees from $\F$, and let $\P_{{lo}}=\P\setminus\P_{hi}$.

  The first two needed properties as well as the convexity of $\F$ are quite immediate.
 By maximality the selected bitiles $t$ are pairwise disjoint, and thus the functions $\Pi_tf$ are pairwise orthogonal, thanks to \eqref{e11}. It follows that
$$
\sum_{T \in \F}|I_T|=\sum_{t} |I_t| \leq 4 \size_f(\P)^{-2} \sum_{t} \|\Pi_tf\|_2^2  \leq 4\,\size_f(\P)^{-2}\|f\|^2_2.
$$
\end{proof}

We can iterate the lemma to obtain
\begin{proposition}[Size decomposition]
\label{2erbyu76i87o}Let $\P$ be a finite convex collection of bitiles. Then
$$
\P= \bigcup_{2^{-n}\le \size_f(\P)}\P_{n} \cup \P_{null}
,$$
such that \begin {itemize}
\item[$\cdot$]$\size_f(\P_{n} )\leq 2^{-n}$:
\item[$\cdot$] $\P_{n}$ is a convex forest with trees $T\in\F_n$ satisfying
\begin{equation}
\label{e19}
\sum_{T\in\F_n}  |I_T|\lesssim 2^{2n}\|f\|^2_2,
\end{equation}

\item [$\cdot$]$\Pi_Pf\equiv 0$ for each $P\in\P_{null}$
\end{itemize}
\end{proposition}
We are now ready for the main line of the argument.

\begin{proof}
[of Theorem \ref{thm:main}]
By using restricted type interpolation and a limiting argument, it suffices to prove
\begin{equation}
\label{e15}
|\{x:|C_{\P}f(x)|\gtrsim \lambda\}|\lesssim\frac1{\lambda^p}|E|
\end{equation}
for each $E\subset \R_+$ with finite measure, each $|f|\le 1_E$, each finite convex $\P\subset \P_{all}$, for each $\lambda>1$, $4<p<\infty$ and also for each $0<\lambda\le 1$, $1<p<\infty$. The implicit constant in the inequality \eqref{e15} will only depend on $p$.

We start with the simpler case $\lambda>1$, $4<p<\infty$. Note that the bound $\size_{f}(\P)\lesssim 1$ follows from Proposition \ref{djhfdftyert}. Let $\P_n, \F_n$ with $2^{-n}\lesssim 1$, be the collections from Proposition \ref{2erbyu76i87o}. For each $T\in \F_n$ with top bitile $P_T\in T$, define the saturation $T^*=\{P\in\P_n:P\le P_T\}$. Note that the trees $T^*$ remain convex, but in general they are not pairwise disjoint. Call $\F_n^*$ the collection of the trees $T^*$, and note that $I_{T^*}=I_T$.

Define the exceptional set
$$F=\bigcup_{2^{-n}\lesssim 1}\bigcup_{T^*\in \F_n^*}\{x: |C_{T^*}f(x)|>\lambda 2^{-n/2}\}.$$
Note that by \eqref{e17} and \eqref{e18} we have
$$|\{x: |C_{T^*}f(x)|>\lambda 2^{-n/2}\}|\lesssim |I_T|(\lambda2^{n/2})^{-p}.$$
Combining this with  \eqref{e19} and $p>4$, we obtain $|F|\lesssim \frac1{\lambda^p}|E|$.

Thus it remains to prove that
$|C_{\P}f(x)|\lesssim \lambda$ on $F^c$. The crucial observation behind this approach to Theorem \ref{thm:main} is that for each $n$ the contribution to each $x$ comes from a single tree $T^*\in \F_n^*$. Indeed, note that the contributing bitiles $P\in\P_n$ satisfy $(x,N(x))\in I_P\times \omega_{P_l}$. Since all these bitiles $P$ contain $(x,N(x))$, they will be pairwise comparable under $\le$. Call $P_x$ the unique maximal bitile among them. Let $T^*_x$ be one of the trees in $\F_n^*$ containing $P_x$. It follows that all the contributing bitiles belong to $T_x^*$. Thus, if $x\notin F$
$$|C_{\P_n}f(x)|=|C_{T_x^*}f(x)|\le \lambda 2^{-n/2}.$$
It further follows by linearity that
$$|C_{\P}f(x)|\le \sum_{2^{-n}\lesssim 1}|C_{\P_n}f(x)|\lesssim \lambda.$$

The case $0<\lambda\le 1$ is very similar but we need an additional exceptional set
$$G:=\{x:M(f)(x)>\lambda^p\}.$$
Since $|G|\lesssim \frac1{\lambda^p}|E|$ and since $W_{P_l}$ is supported on $I_P$, it is enough to prove \eqref{e15} with $\P$ restricted to those bitiles such that $P\not\subset G$. Another application of Proposition \ref{djhfdftyert} shows that $\size_{f}(\P)\lesssim \lambda^p$. Let $\P_n, \F_n$ with $2^{-n}\lesssim \lambda^p$, be the collections from Proposition \ref{2erbyu76i87o} corresponding to our new $\P$. Define
$$F=\bigcup_{2^{-n}\lesssim \lambda^p}\bigcup_{T^*\in \F_n^*}\{x: |C_{T^*}f(x)|>\lambda^{1/2} 2^{-\frac{n}{2p}}\}.$$
As before \eqref{e17} and \eqref{e18} imply
$$|\{x: |C_{T^*}f(x)|>\lambda^{1/2} 2^{-\frac{n}{2p}}\}|\lesssim |I_T|(\frac{2^{-n}}{\lambda^{1/2} 2^{-\frac{n}{2p}}})^{p/p-1}.$$
We immediately get that
$|F|\lesssim |E|\le \frac1{\lambda^p}|E|.$ Note also that if $x\notin F$
$$|C_{\P}f(x)|\le \sum_{2^{-n}\lesssim \lambda^p}|C_{\P_n}f(x)|\le \sum_{2^{-n}\lesssim \lambda^p}\lambda^{1/2} 2^{-\frac{n}{2p}}\lesssim \lambda.$$
\end{proof}

\section{Mass and Fefferman's argument}
\label{FEFE}
The argument in this section is a translation to the Walsh case of Fefferman's proof \cite{Fe}. The key tool used in this proof is  {\em mass}.

\begin{definition}
The mass of a convex collection $\P$ of bitiles is defined as
$$\mass(\P)=\sup_{P\in \P}\frac{|E(P)|}{|I_P|}$$
where $E(P)=I_P\cap N^{-1}(\omega_P)$.
\end{definition}
In some sense the mass of a single bitile $P$ measures (or better said, it puts an upper bound on- since $\omega_{P_l}\subset \omega_P$)  how much $P$ contributes to $C_Pf(x)$. Indeed, it suffices to note that for $1\le p\le\infty$ we have
$\|C_Pf\|_p\le \mass(P)^{1/p}\|f\|_p.$ We will next extend this inequality to the case of trees and then to a special type of forests.

We have the following analogue of Lemma \ref{sizelemma}. We will restrict attention to the bitiles spatially supported in $[0,1]$
$$\P_{[0,1]}:=\{P\in\P_{all}: I_P\subset [0,1]\}.$$

\begin{lemma}
 \label{masslemma}
Let $\P$ be a finite convex collection of bitiles in $\P_{[0,1]}$. Then
$
\P= \P_{hi} \cup \P_{{lo}}
,$
such that \begin {itemize}
\item[$\cdot$]both $\P_{{lo}}$ and $ \P_{hi} $ are convex
\item[$\cdot$]$\mass(\P_{{lo}} )\leq \frac12\mass(\P)$:
\item[$\cdot$] $\P_{hi}$ is a convex forest with trees $T\in\F$ satisfying
$$
\sum_{T\in\F}  |I_T|\lesssim \mass(\P)^{-1}.
$$
\end{itemize}
\end{lemma}
\begin{proof} This is another recursive procedure. Set $\P_{{stock}}:=\P$. Select a maximal  bitile  $t \in \P_{{stock}}$ such that
$$
\mass(t) > \frac{\mass(\P)}{2}
$$
Define
$$
T(t) = \{P \in \P_{{stock}}: P \le t\}.
 $$
and note that since $\P_{{stock}}$ is convex, both $\P_{{stock}}\setminus T(t)$ and the tree $T(t)$ will be  convex. Add $T(t)$ to the family $\F$. Reset the new value $\P_{{stock}}:=\P_{{stock}}\setminus T(t)$, and restart the procedure.

The algorithm is over when there is no $t$ to be selected. Then define $ \P_{hi} $ to consist of the union of  all bitiles in all the trees from $\F$, and let $\P_{{lo}}=\P\setminus\P_{hi}$.

  The first two needed properties as well as the convexity of $\F$ are quite immediate.
 By maximality the selected bitiles $t$ are pairwise disjoint, and thus the sets $E(t)$ will also be pairwise disjoint. Since  $E(t)\subset [0,1]$, it follows that
\begin{equation}
\label{iod fhn rygui}
\sum_{T \in \F}|I_T|=\sum_{t} |I_t| \leq 2 \mass(\P)^{-1} \sum_{t} |E(t)|  \leq 2\mass(\P)^{-1}.
\end{equation}
\end{proof}

Note that the mass of any collection is trivially bounded by 1. We can iterate the lemma to obtain the following analogue of Proposition \ref{2erbyu76i87o}

\begin{proposition}[Mass decomposition]
\label{massdecomp}Let $\P\subset \P_{[0,1]}$ be a finite convex collection of bitiles. Then
$$
\P= \bigcup_{n\ge 0}\P_{n} \cup \P_{null}
,$$
such that \begin {itemize}
\item[$\cdot$]$\mass (\P_{n} )\leq 2^{-n}$:
\item[$\cdot$] $\P_{n}$ is a convex forest with trees $T\in\F_n$ satisfying
$$
\sum_{T\in\F_n}  |I_T|\lesssim 2^{n},
$$

\item [$\cdot$]$C_Pf\equiv 0$ for each $P\in\P_{null}$ and each $f$
\end{itemize}
\end{proposition}

It turns out that we have not only $L^1$ but also dyadic $BMO_{\Delta}$ control for $N_{\F_n}$. Recall that
$$\|f\|_{BMO_{\Delta}}=\sup_{I\text{ dyadic }}\frac1{|I|}\int_I\left |f(x)-\frac{1}{|I|}\int_If \right|dx$$
\begin{proposition}
We have
$$\|N_{\F_n}\|_{BMO_{\Delta}}\lesssim 2^n.$$
In particular, for each $I\in\D_+$ and $\lambda>0$
$$|\{x:\sum_{T\in \F_n:I_T\subset I}1_{I_T}(x)\ge C\lambda 2^{n}\}|\le e^{-\lambda}|I|,$$
where $C$ is large enough.
\end{proposition}
\begin{proof}
One can easily check that
$$\|N_{\F_n}\|_{BMO_\Delta}\lesssim \sup_{I\in\D_+}\frac{\sum_{T\in \F_n:I_T\subset I}|I_T|}{|I|}.$$
This combined with \eqref{iod fhn rygui} will imply our first inequality.

The second one is just a consequence of John-Nirenberg's inequality.
\end{proof}

Given a convex tree $T$ with top data $(\xi_T,I_T)$ , define $\J_T$ to be the collection of all maximal dyadic intervals $J\subset I_T$ such that $J$ contains no $I_P$ with $P\in T$. For $J\in\J_T$ define
$$G_J=J\cap (\bigcup_{P\in T:J\subset I_P} E(P)).$$
We first observe that the intervals of $\J_T$ form a partition of $I_T$. Note also that if $J$ intersects some $I_P$ with $P\in T$, then  the convexity of $T$ forces the dyadic parent of each $J\in \J_T$ to equal $I_{P(J)}$, for some $P(J)\in T$. Thus
 $G_J\subset E(P(J))$, which implies  the
following crucial Carleson measure type estimate
\begin{equation}
\label{Carmeasmass}
|G_J|\le 2\mass(T)|J|,
\end{equation}for each $J\in\J_T$.
We also need to observe that $C_Tf$ is supported on $\bigcup_{J\in\J_T}G_J$.

\begin{proposition}[Tree estimate]
\label{treeestimate}
Let $T$ be a convex tree. Then for each $1<p<\infty$
$$\|C_Tf\|_{p}\lesssim (\mass(T))^{1/p}\|f\|_p.$$
\end{proposition}
\begin{proof}Assume first that $T$ is overlapping.
Note that for each $J\in\J_T$ and $x\in J$, if $C_Tf(x)$ is nonzero   then
$|C_Tf(x)|=\frac{1}{|I_P|^{1/2}}|\langle f, W_{P_u}\rangle|$
for some $P\in T$ with $J\subset I_P$. Thus
$|C_Tf(x)|\le \inf_{y\in J}M(f)(y)$. We conclude that
$$\int|C_Tf(x)|^pdx=\sum_{J\in\J_T}\int_{G_J}|C_Tf(x)|^pdx\lesssim \sum_{J\in\J_T}|G_J|\inf_{y\in J}[M(f)(y)]^p$$$$\le 2\mass(T)\sum_{J\in\J_T}\int_{J}|Mf(x)|^pdx=2\mass(T)\|Mf\|_{L^p(I_T)}^p\lesssim \mass(T)\|f\|_p^p$$

On the other hand, if $T$ is lacunary we reason like in the proof of Lemma \ref{furfy58t78t585igutgutih} to write
$$C_Tf(x)=\F_W^{-1}(1_{\omega}\F_W(O_Tf))(x),$$
where the interval $\omega$ is any interval which contains all $\omega_{P_u}$ with $x\in I_P$. Recall however that $|I_P|\ge 2|J|$ for all such $P$ and hence, since $\xi_T\in\omega_{P_l}$, all $\omega_{P_u}$ will be contained in the interval $[\xi_T,\xi_T+|J|^{-1}]$. Thus
$$|C_Tf(x)|\le \frac1{|J|}\int_{x\oplus[0,|J|]}|O_Tf|(y)dy\lesssim \inf_{y\in J}M(O_Tf)(y).$$
A repetition of the argument from the overlapping case combined with Lemma \ref{lemasingtree1} ends the proof.
\end{proof}
The proof of this proposition shows that for each $J\in\J_T$
\begin{equation}
\label{jvgirtg782r]o3olfkjtiohj}
\|C_Tf\|_{L^\infty(J)}\le \inf_{x\in J}M(V_Tf)(x)
\end{equation}
where $V_Tf=\Pi_Tf$ if $T$ is overlapping and $V_Tf=O_T(\Pi_Tf)$ is $T$ is lacunary.

\begin{definition}We call $\F$ a Fefferman forest if two bitiles from any two distinct trees in $\F$ are pairwise disjoint.
\end{definition}
It turns out that outside a small exceptional set, each forest $\F$ can be written as a disjoint union of a small number of Fefferman forests.

\begin{lemma}[The Fefferman trick]\label{Fefftrick}
Assume  $\P$ is a convex set in $\P_{[0,1]}$, which is organized as a forest $\F$ satisfying
$$\|N_{\F}\|_{BMO_\Delta}\lesssim 2^{n}$$
and
$$\|N_\F\|_{L^1}\lesssim 2^{n}.$$
Then, for each $K>1$ there is an exceptional set $F$ with $|F|\le e^{-K}$ such that
\begin{itemize}
\item \begin{equation}\label{e20}\|\sum_{T\in\F:I_T\not\subset F}1_{I_T}\|_{\infty}\lesssim  K2^n\end{equation}
\item the bitiles in $\P_F:=\bigcup_{T\in\F:I_T\not\subset F}\{P:P\in T\}$ can be partitioned into  $O(\log (K2^n))$ Fefferman forests
\end{itemize}
\end{lemma}
\begin{proof}Define $F=\{x: \sum_{T\in \F}1_{I_T}(x)>CK2^n\}$, for large enough $C$. Note that $|F|\le e^{-K}$ by John-Nirenberg's inequality. Call $\F_F:=\{T\in\F:I_T\not\subset F\}$. Then \eqref{e20} is immediate. Let $P_T$ be any top bitile for $T$. Define for each $l\in \N$ with $1\le 2^l\le CK2^n$
$$\P_F^l:=\{P\in\P_F:2^l\le \#\{T\in \F_F:P\le P_T\}<2^{l+1}\}.$$
It remains to prove that $\P_F^l$ is a Fefferman forest. Note that $\P_F^l$ is convex, in part because $\P$ is convex. For each maximal element $t\in \P_F^l$, let $T(t)=\{P\in \P_F^l:P\le t\}$. Obviously each tree $T(t)$ is convex. Assume for contradiction $P\le P'$ for some $P\in T(t), P'\in T(t')$ with $t\not=t'$. Then $P\le t'$, in addition to $P\le t$. Thus $I_t$ intersects $I_{t'}$. But then,  since $t$ and $t'$ are pairwise incomparable under $\le$, the sets $\{T\in \F_F:t\le P_T\}$ and $\{T\in \F_F:t'\le P_T\}$ must be pairwise disjoint. Note that $$\{T\in \F_F:t\le P_T\}\cup\{T\in \F_F:t'\le P_T\}\subset\{T\in \F_F:P\le P_T\},$$
and this will force the contradiction $2^l+2^l<2^{l+1}$.
\end{proof}

We will repeatedly use the fact that if $T,T'$ are trees in a Fefferman forest then $C_Tf$ and $C_{T'}f$ are disjointly  supported while $\Pi_Tf$ and $\Pi_{T'}f$ are orthogonal.

\begin{proposition}[Forest estimate]
\label{forestesftttttimate}
Let $\F$ be a Fefferman forest.
Then for each $1<p<\infty$ there exists $\delta(p)<\frac1p$ such that
$$\|C_\F f\|_{p}\lesssim (\mass(\F))^{1/p}(\|N_\F\|_{\infty})^{\delta(p)}\|f\|_p.$$
When $p\ge 2$ we can take $\delta(p)=0$.
\end{proposition}
\begin{proof}

Let first $p\ge 2$. Consider the vector valued operator $Vf=(\Pi_Tf)_{T\in\F}$. Proposition \ref{hfuyhrufywiei3u884u8} shows that $\|V\|_{L^{\infty}\to l^{\infty}(L^\infty)}\lesssim 1$. On the other hand, the pairwise orthogonality of $\Pi_Tf$ implies the bound $\|V\|_{L^{2}\to l^{2}(L^2)}\lesssim 1$. Interpolation \cite{BePa} now gives
\begin{equation}
\label{elpbencanz}
\|V\|_{L^{p}\to l^{p}(L^p)}\lesssim 1.
\end{equation}
Using Proposition \ref{treeestimate}  we get
$$\|C_\F f\|_p^p=\sum_{T\in \F}\|C_Tf\|_p^p\lesssim \sum_{T\in \F}\mass(T)\|\Pi_T f\|_p^p\le \mass(\F)\|f\|_p^p.$$

Assume next that $p<2$. Split $\F$ into $\|N_\F\|_{\infty}$ forests $\F_k$ so that for each $k$ and each $T,T'\in \F_k$ we have $I_T\cap I_{T'}=\emptyset$. Note that we have  as before
$$\|C_\F f\|_p^p=\sum_k\sum_{T\in \F_k}\|C_Tf\|_p^p\lesssim \sum_k\sum_{T\in \F_k}\mass(T)\|f1_{I_T}\|_p^p$$$$\le \mass(\F)\sum_k\|f\|_p^p=\mass(\F)\|N_\F\|_{\infty} \|f\|_p^p.$$
To close the argument, interpolate the $L^p$ and $L^2$ bounds for the operator $f\mapsto C_\F f$.
\end{proof}

We are now ready to prove the following variant of Theorem \ref{thm:main}. Note that this weaker version is enough to prove Theorem \ref{thm:Bil}, via the standard approximation argument. Then, in the case $1<p\le 2$ one can actually use Stein's Continuity Principle \cite{Ste} to show that Theorem \ref{thm:Bil}  implies the stronger Theorem \ref{thm:main}. The author is not aware of any substitute argument which proves the same implication for $p>2$. However, in Section \ref{victor} we present a recent delicate refinement of Fefferman's argument due to Lie \cite{Lie2} which closes this gap.
\begin{theorem}
\label{sub:main}
For each $1<q<p<\infty$ we have
$$\|C_{\P_{[0,1]}} f\|_{q}\lesssim \|f\|_p$$
\end{theorem}
\begin{proof}
By invoking a limiting argument, it suffices to prove the estimate with $\P_{[0,1]}$ replaced by a finite convex collection  $\P\subset \P_{[0,1]}$, as long as the implicit constant in the inequality is independent of $\P$.  Apply  Proposition \ref{massdecomp} to $\P$ to get $\P_n,\F_n$. Apply Lemma \ref{Fefftrick} to each $\F_n$ with $K=K_n:=(n+1)L$, $L\ge 2$ to be determined later. We get exceptional sets $|F_n|\le e^{-L(n+1)}$ and the partitions of $\F_{F_n,n}:=\{T\in\F_n:I_T\not\subset F_n\}$ into $O(n+\log L)$ Fefferman forests $\F_{n,k}$. Define $F=\cup_{n}F_n$ and note that $|F|\lesssim e^{-L}$. Let also $\P^*$ be the bitiles in the trees from  $\cup_n\F_{F_n,n}$. Note that $C_\P f(x)=C_{\P^*}f(x)$ on $F^c$. By Proposition \ref{forestesftttttimate} and linearity we get
$$\|C_{\P}f\|_{L^p(F^c)}\le\sum_{n}\sum_k\|C_{\F_{n,k}}f\|_{L^p([0,1])}\lesssim \sum_n2^{-n/p}(n+\log L)(n2^n L)^{\delta(p)}\|f\|_p\lesssim L^{1/p}\|f\|_p.$$
Thus, for each $\lambda>1$
$$|\{x\in[0,1]:|C_{\P}f(x)|>\lambda\}|\lesssim L\left(\frac{\|f\|_p}{\lambda}\right)^{p}+e^{-L}.$$
By optimizing $L$ we get that for each $s<p$
$$|\{x\in[0,1]:|C_{\P}f(x)|>\lambda\}|\lesssim \left(\frac{\|f\|_p}{\lambda}\right)^{s}.$$
A simple integration argument finishes the proof of the theorem.

\end{proof}

\section{Combining mass and size: The Lacey-Thiele argument}
\label{sec:Lac-Thi}
The Fourier version of this argument is due to Lacey and Thiele \cite{LT-C}. It uses both mass and size and thus it combines elements of the two proofs we have seen in earlier sections. The interplay between mass and size made this approach particularly well suited for applications to the problem of singular integrals along vector fields \cite{LaLi1}, \cite{Ba}, \cite{BT},  since it  opened the door to the use of Kakeya type maximal functions.

The definition of $\size(\P)$ will remain the same as in Section \ref{sec:3}. Let $F$ be a finite measure subset of $\R_+$. We modify slightly the definition of mass.
\begin{definition}
The mass of a convex collection $\P$ of bitiles relative to $F$ is defined as
$$\mass_F(\P)=\sup_{P\in \P}\frac{|E_F(P)|}{|I_P|}$$
where now $E_F(P)=F\cap I_P\cap N^{-1}(\omega_P)$.
\end{definition}
We have the following version of Proposition \ref{massdecomp}.

\begin{proposition}[Mass decomposition]
\label{massdecompver2}Let $\P\subset \P_{all}$ be a finite convex collection of bitiles. Then
$$
\P= \bigcup_{m\ge 0}\P_{m}^* \cup \P_{null}^*
,$$
such that \begin {itemize}
\item[$\cdot$]$\mass_F (\P_{m}^* )\leq 2^{-m}$:
\item[$\cdot$] $\P_{m}^*$ is a convex forest with trees $T\in\F_m^*$ satisfying
$$
\sum_{T\in\F_m^*}  |I_T|\lesssim 2^{m}|F|,
$$

\item [$\cdot$]$C_Pf\equiv 0$ for each $P\in\P_{null}^*$ and each $f:\R_+\to\C$.
\end{itemize}
\end{proposition}

\begin{proposition}[Tree estimate]
\label{sceruyfr45f832p}Let $T$ be a convex tree. Then for each $f\in L^2(\R_+)$
$$|\langle C_T f, 1_F\rangle|\lesssim |I_T|\mass_F(T)\size_f(T).$$
\end{proposition}
\begin{proof}
For $J\in\J_T$ define as before
$G_J=J\cap(\bigcup_{P\in T:J\subset I_P} E_F(P)).$
Recall that $C_Tf1_F$ restricted to $J$ is supported on $G_J$. We then estimate like in the proof of Proposition \ref{treeestimate}
$$|\langle C_T f, 1_F\rangle|\le \sum_{J\in\J_T}\int_{G_J}|C_T f|\le \mass_F(T)\sum_{J\in\J_T}\int_JM(V_T(f))\le $$$$\mass_F(T)\sum_{J\in\J_T}|J|^{1/2}(\int_J|M(V_T(f))|^2)^{1/2}\le \mass_F(T)|I_T|^{1/2}\|V_T(f)\|_2,$$
where $V_Tf=\Pi_Tf$ if $T$ is overlapping and $V_T=O_T(\Pi_Tf)$ is $T$ is lacunary. The result now follows from \eqref{e18} and Lemma \ref{lemasingtree1}.
\end{proof}
\begin{proof}
[of Theorem \ref{thm:main}] By using restricted type interpolation and a limiting argument, it suffices to prove that given $1<p<\infty,p\not=2$ with dual exponent $p'$ and given finite measure subsets  $E,G$ of $\R_+$, there exists $F\subset G$ with $|F|\ge |G|/2$ such that
\begin{equation}
\label{e21}
|\langle C_\P f, 1_F\rangle|\lesssim |E|^{1/p}|F|^{1/p'}
\end{equation}
for  each $|f|\le 1_E$ and each finite convex $\P\subset \P_{all}$.

We analyze first the case $p>2$, when we can take $F=G$. Note that $\size_f(\P)\lesssim 1$. Apply to $\P$ Propositions \ref{2erbyu76i87o} and \ref{massdecompver2} to get $\P_n,\P_m^{*},\F_n,\F_m^*$ for $m\ge 0$ and $2^{-n}\lesssim 1$. Define $\P_{n,m}:=\P_n\cap \P_m^*$. We  organize the bitiles in $\P_{n,m}$ into a forest in two different way. Call $\F_{n,m}$ the trees $T\cap \P_m^*$ with $T\in\F_n$ and call $\F_{n,m}^*$ the trees $T\cap \P_n$ with $T\in\F_m^*$. Note that the resulting trees are convex and
$$\sum_{T\in \F_{n,m}}|I_T|\lesssim 2^{2n}|E|$$
$$\sum_{T\in \F_{n,m}^*}|I_T|\lesssim 2^{m}|F|.$$
Thus, using Proposition \ref{sceruyfr45f832p} with the two partitions above, we get
$$|\langle C_{\P_{n,m}} f, 1_F\rangle|\lesssim 2^{-n}2^{-m}\min(2^{2n}|E|, 2^{m}|F|).$$
We conclude that
$$|\langle C_\P f, 1_F\rangle|\le \sum_{2^{-n}\lesssim 1,m\ge 0}|\langle C_{\P_{n,m}} f, 1_F\rangle|\lesssim \sum_{n,m}2^{-n}2^{-m}(2^{2n}|E|)^{1/p}(2^{m}|F|)^{1/p'}\lesssim |E|^{1/p}|F|^{1/p'}.$$

Consider now the case $p<2$. We define
$$F:=G\setminus\{x:M(1_E)(x)>10\frac{|E|}{|G|}\}.$$ Note that $|F|\ge |G|/2$. Call $\P^*=\{P\in\P:I_P\not\subset F^c\}$ and note that
$\langle C_\P f, 1_F\rangle=\langle C_{\P^*} f, 1_F\rangle$. Proposition \ref{djhfdftyert} implies the bound $\size_f(\P^*)\lesssim \frac{|E|}{|F|}$.

Proceeding as before, with the two partitions this time for $\P^*$, we obtain
$$|\langle C_{\P^*} f, 1_F\rangle|\le \sum_{2^{-n}\lesssim \frac{|E|}{|F|},m\ge 0}|\langle C_{\P_{n,m}} f, 1_F\rangle|\lesssim $$$$\sum_{2^{-n}\lesssim \frac{|E|}{|F|},m\ge 0}2^{-n}2^{-m}(2^{2n}|E|)^{1/p'}(2^{m}|F|)^{1/p}\lesssim |E|^{1/p}|F|^{1/p'}.$$
\end{proof}

\section{A direct proof of strong $L^p$ bounds}
\label{victor}
All three arguments presented earlier rely on proving weak type bounds and using interpolation. Moreover, as observed in Section  \ref{FEFE},  Fefferman's argument does not imply strong type $L^p$ bounds for the operator $C_{\P_{[0,1]}}$ when $p\ge 2$. The recent argument of Lie \cite{Lie2} proves that
\begin{equation}
\label{fjrgh845788-0990jgikbjg}
\|C_{\P_{[0,1]}}f\|_{p}\lesssim \|f\|_{L^p[0,1]}
\end{equation}
directly, without any interpolation that restricts $f$ to special classes of functions (such as sub-characteristic). The case $p=2$ will require no interpolation in the argument, while the case $p\not=2$ will rely on vector valued interpolation for various operators relevant to the proof.  We will restrict attention to $p\ge2$.

This argument, while also interesting in itself, has been developed in \cite{Lie2} in order to solve a conjecture of Stein on the boundedness of the quadratic Carleson operator.

We will keep the notation for mass from Section \ref{FEFE} and will restrict attention to the bitiles in $\P_{[0,1]}$. Recall (see Lemma \ref{Fefftrick}) that Fefferman's approach consisted of decomposing every forest into a small number of Fefferman forests, outside a small exceptional set. The lemma below shows how to iterate the Fefferman trick inside the exceptional set, until all bitiles are exhausted. Each stage of the iteration creates more layers of Fefferman forests, but the size of their spatial support will get exponentially smaller.

\begin{lemma}
\label{l1}
Let $\P\subset\P_{[0,1]}$ be a finite convex collection of bitiles. Then we have the following partitions
\begin{equation}
\label{e1}
\P=\bigcup_{k\ge 0}{\P_k}
\end{equation}\begin{equation}
\label{e2}
\P_k=\bigcup_{m\ge 1}\Q^{m}_k\end{equation}\begin{equation}
\label{e3}
\Q^m_k=\bigcup_{1\le n\lesssim k}\F^{m,n}_k
\end{equation}
where each $\F^{m,n}_k$ is a Fefferman forest such that
\begin{equation}
\label{e77}
\|N_{\F^{m,n}_k}\|_{\infty}\lesssim k2^k
\end{equation}
and
\begin{equation}
\label{e7}
\mass(\P\setminus\bigcup_{i=0}^{k-1}\P_i)\le 2^{1-k}.
\end{equation}
Moreover, there are sets $(E_{k}^m)_{m,k\ge 0}$ which are finite disjoint unions $\J_k^m$ of  dyadic intervals such that
\begin{equation}
\label{e4}
E_{k}^{m+1}\subset E_{k}^m\text{ and }|E_{k}^{m+1}\cap J|\le e^{-10k}|J|, \text{ for }J\in \J_k^m
\end{equation}
\begin{equation}
\label{e5}
P\in\Q^{m+1}_k \text{ implies }I_P\subset E_{k}^{m},I_P\not\subset E_{k}^{m+1}
\end{equation}
\begin{equation}
\label{e6}
P\in\Q^{m+1}_k \text{ implies }|I_P\cap E_{k}^{m+2}|\le e^{-10k}|I_P|
\end{equation}
\end{lemma}
\begin{proof}
Set $\P_0:=\emptyset$. Assume that given  $k>0$, each $\P_i$ has been defined for each $0\le i\le k-1$ so that it satisfies \eqref{e2} - \eqref{e6}.

Let us see how to define $\P_k$. For the rest of the argument set  $\P_{stock}=\P\setminus(\cup_{0\le i\le k-1}\P_i)$. Let
$\P_{k,max}^{0}$ be the maximal bitiles in the set
$$\{P\in\P_{stock}:\mass(P)\ge 2^{-k}\}.$$
Define $E_{k}^0:=[0,1]$,
$$\P_k^0=\{P\in \P_{stock}:P\le P'\text{ for some }P'\in \P_{k,max}^{0}\}$$
$$N_{k}^0(x)=\sum_{P\in\P_{k,max}^0}1_{I_P}(x)$$
and the exceptional set
$$E_{k}^1=\{x: N_k^0(x)\ge C_1k2^k\}$$
where $C_1$ is a large enough constant to be determined later. Note that since the sets $E_P$ are pairwise disjoint for $P\in \P_{k,max}^{0}$, it follows that  $\|N_k^0\|_{BMO_\Delta}\le 2^{k}$. This and John-Nirenbeg in turn imply \eqref{e4} for $m=0$.
Define
$$\Q_{k}^1=\{P\in\P_{k}^0:I_P\not\subset E_{k}^1\}$$
and note that \eqref{e5} holds for $m=0$.

We continue the construction of the sets $E_k^{m}$, $\Q_k^m$ inductively. Fix $m_0\ge 1$. Assume $E_k^{0},\ldots E_{k}^{m_0}$ and $\Q_k^{1},\ldots, \Q_k^{m_0}$ have been constructed so that \eqref{e4}-\eqref{e5} hold for each $0\le m\le m_0-1$ and so that \eqref{e6} holds for each $m\le m_0-2$.
Define
$\P_{k,max}^{m_0}$ to be the maximal bitiles in the set
$$\{P\in\P_{stock}:I_P\subset E_{k}^{m_0},\;\mass(P)\ge 2^{-k}\}.$$
Let also
$$\P_k^{m_0}=\{P\in \P_{stock}:P\le P'\text{ for some }P'\in \P_{k,max}^{m_0}\},$$
$$N_{k}^{m_0}(x)=\sum_{P\in\P_{k,max}^{m_0}}1_{I_P}(x)$$
and the exceptional set
$$E_{k}^{m_0+1}=\{x: N_k^{m_0}(x)\ge C_1k2^k\}.$$
Finally, define
$$\Q_{k}^{m_0+1}=\{P\in\P_{k}^{m_0}:I_P\not\subset E_{k}^{m_0+1}\}.$$
Note that \eqref{e4}-\eqref{e5} hold when $m=m_0$, for the same reasons as before. To check \eqref{e6} for $m=m_0-1$, let $P\in \Q^{m_0}_k$. Since $I_p\not\subset E_k^{m_0}$, we have that either $I_P\cap E_k^{m_0}=\emptyset$, or $I_P\cap E_k^{m_0}$ is a finite disjoint union of dyadic intervals $J\in \J_k^{m_0}$. Note that John-Nirenberg's inequality guarantees that for each $J\in \J_k^{m_0}$ we have
$$|J\cap E_k^{m_0+1}|\le e^{-10k}|J|,$$
which immediately implies \eqref{e6} for $m=m_0-1$.

It is easy to see that each $\Q_k^{m}$ is convex. It can be organized into a forest of trees with tops in $$\{P\in\P_{k,max}^{m-1}:I_P\not\subset E_k^{m}\}.$$ Note that the counting function has a favorable $L^\infty$ bound
$$\|\sum_{P\in\P_{k,max}^{m-1}:I_P\not\subset E_k^{m}}1_{I_P}\|_{\infty}\le C_1k2^k.$$
The forests $\F^{m,n}_k$ are now obtained via the Fefferman trick.

It is immediate that the collections $\Q_k^m$ are pairwise disjoint for fixed $k$, because of \eqref{e5}. Note that the algorithm (for fixed $k$) will end with a finite value of $m$. We set $\Q_k^m$ to be empty for all larger values of $m$ and define $\P_k=\bigcup_{m\ge 1}\Q^{m}_k$. Finally, observe that when the algorithm for fixed $k$ is over, there can not be any $P$ left in $\in\P_{stock}\setminus\P_k$ with $\mass(P)\ge 2^{-k}$. Indeed, note that $I_P\subset E_k^0$ trivially. While $I_P\subset E_k^m$ the algorithm will continue to run. But by \eqref{e4}, there should exist $m$ such that $I_P\subset E_{k}^{m},I_P\not\subset E_{k}^{m+1}$. Thus, if $\mass(P)\ge 2^{-k}$, then $P$ gets automatically selected in $\Q_{k}^{m+1}$.
\end{proof}

A distinct feature of the approach in this section is the almost orthogonality between a function with small support and a function locally constant on this support.

\begin{lemma}[Almost orthogonality]
\label{lemaortho}
Let $\J$ be a collection of pairwise disjoint  intervals and let $f:\cup_{J\in\J}J\to\C$ such that $|f|$ is constant on each $J\in\J$. Assume $E_J\subset J$ satisfies $|E_J|\le \alpha|J|$ for each $J\in\J$. Then for each $g:\cup_{J\in\J}E_J\to\C$ we have
$$|\langle f, g\rangle|\le \alpha^{1/2}\|f\|_2\|g\|_2.$$
\end{lemma}
\begin{proof}
$$|\langle f, g\rangle|\le \sum_J \int_{E_J}|f||g|\le \sum_J|E_J|^{1/2}\sup_{x\in J}|f(x)|\|g\|_{L^2(J)}$$$$\le \alpha^{1/2}\sum_J\|f\|_{L^2(J)}\|g\|_{L^2(J)}\le \alpha^{1/2}\|f\|_2\|g\|_2$$
\end{proof}
\begin{proposition}[Almost orthogonality between forests]
\label{proportho}
Fix $k$. For each $m\ge 0$, let $\F^m$ denote one of the forests $\F^{m,n}_k$ from Lemma \ref{l1}. Define $$\Pi_{\F^m}f=\sum_{T\in\F^m}\Pi_Tf.$$ Then for each $m'\ge m+2$
$$|\langle \Pi_{\F^m}f, \Pi_{\F^{m'}}f\rangle|\lesssim e^{-(m'-m-2)k/2}\|f\|_2^2$$
\end{proposition}
\begin{proof}
Note that due to \eqref{e77}
 we can split $\F^m$ into layers $\F^m=\bigcup_{l\lesssim k2^k}\F^m(l)$, so that for each $l$ the trees $T$ in $\F^m(l)$ have pairwise disjoint intervals $I_T$.

Arguing like in the proof of Proposition \ref{hfuyhrufywiei3u884u8}
we first observe that for each $x$ we have $\sum_{T\in\F^m(l)}\Pi_Tf(x)=\Pi_Pf(x)$, where  $P$ is the minimal bitile in $\cup_{T\in\F^m(l)}T$ such that $x\in I_P$. Second, by \eqref{e5}, for  each interval $J\in \J_k^m$ and each $P\in \cup_{T\in\F^m(l)}T$, $J$ can not contain either the right or the left halves of $I_P$. Third, \eqref{e54757897985960869078089=-} guarantees that each $|\Pi_Pf|$ is constant on both the right and the left halves of $I_P$. We now conclude that
the function $|\sum_{T\in\F^m(l)}\Pi_Tf|$ is constant on each $J\in \J_k^m$. Applying Lemma \ref{lemaortho} to $(\sum_{T\in\F^m(l)}\Pi_Tf)1_{E_k^m}$ and $\Pi_{\F^{m'}}f$, by virtue of \eqref{e4} we get
$$|\langle \Pi_{\F^m}f, \Pi_{\F^{m'}}f\rangle|\le \sum_{l}|\langle \Pi_{\F^m(l)}f, \Pi_{\F^{m'}}f\rangle|$$
$$\lesssim k2^{k}e^{-10(m'-m-1)k/2}\|\Pi_{\F^{m'}}f\|_2\|\Pi_{\F^{m}}f\|_2\lesssim e^{-(m'-m-2)k/2}\|f\|_2^2.$$

\end{proof}

\begin{lemma}[Schur's test]Let $T_m$ be a sequence of operators on $L^2([0,1])$ with adjoints $T_m^*$ such that for each $m,m'$
$$\|T_{m'}T_m^*\|_{2\mapsto 2}\le c_{|m'-m|}$$
where
$$\sum_{n\ge 0}c_n=c<\infty.$$
Then
$$\sum_m\|T_mf\|_2^2\le c\|f\|_2^2.$$
\end{lemma}
\begin{proof}
We have
$$\sum_m\|T_mf\|_2^2=\sum_m\langle T_{m'}^*T_m f, f\rangle\le \|\sum_mT_m^*T_mf\|_2\|f\|_2$$
$$=(\sum_{m,m'}\langle T_m^*T_mf, T_{m'}^*T_{m'}f\rangle)^{1/2}\|f\|_2\le (\sum_{m,m'}\|T_{m'}T_m^*\|_{2\mapsto 2}\|T_mf\|_2\|T_{m'}f\|_2)^{1/2}\|f\|_2.$$
It now suffices to note that this can further be bounded via Cauchy-Schwartz by
$$c^{1/2}(\sum_m\|T_mf\|_2^2)^{1/2}\|f\|_2.$$
\end{proof}

\begin{corollary}\label{corortho}With the notation in Proposition \ref{proportho} we have for each $2\le p<\infty$
$$\sum_{m}\|\Pi_{\F^m}f\|_p^p\lesssim (k2^k)^{p-2}\|f\|_p^p.$$
\end{corollary}
\begin{proof}
To get the $p=2$ case apply Schur's test to the families of operators $T_m=\Pi_{\F^{3m+i}}$ with $i\in\{0,1,2\}$. Note that $\Pi_{\F}^*f=\overline{\Pi_\F(\bar{f})}$ for each forest $\F$.

We next get an $L^\infty$ bound. Since by \eqref{e77} each forest consists of $O(k2^k)$ layers, Proposition \ref{hfuyhrufywiei3u884u8} implies the bound $\|\Pi_{\F^m}f\|_{\infty}\lesssim k2^k\|f\|_{\infty}$. Then we use  interpolation \cite{BePa} for the vector valued operator $Vf=(\Pi_{\F^m}f)_{m\ge 1}$ like in the proof of Proposition \ref{forestesftttttimate}.
\end{proof}

\begin{lemma}\label{cororthocmnheri6t783489ry958y}With the notation in Proposition \ref{proportho} we have for each $2\le p<\infty$
$$\sum_{m}\|C_{\F^m}f\|_p^p\lesssim 2^{-k}\|f\|_p^p.$$
\end{lemma}
\begin{proof}
Using that $\F^m$ is a Fefferman forest and the tree estimate Proposition \ref{treeestimate} we get
$$\|C_{\F^m}f\|_p^p=\sum_{T\in \F^m}\|C_{T}f\|_p^p\lesssim 2^{-k}\sum_{T\in \F^m}\|\Pi_{T}f\|_p^p$$
Consider the vector valued operator $Vf=(\Pi_Tf)_{T\in\cup_m\F^m}$. Note that
$$\|Vf\|_{L^2(l^2)}=(\sum_{m}\|\Pi_{\F^m}f\|_2^2)^{1/2}\lesssim \|f\|_2,$$
by Corollary \ref{corortho}, and also that
$$\|Vf\|_{L^\infty(l^\infty)}=\sup_{T\in\cup_m\F^m}\|\Pi_{T}f\|_\infty\lesssim \|f\|_\infty,$$
by virtue of Proposition \ref{hfuyhrufywiei3u884u8}. Interpolation \cite{BePa} finishes the proof.
\end{proof}

The proof of \eqref{fjrgh845788-0990jgikbjg} when $p\ge 2$ will now immediately follow from the next estimate, by invoking the triangle inequality.

\begin{proposition}Let $2\le p<\infty$.
For each $f\in L^p([0,1])$ we have
$$\|C_{\P_k}f\|_p\lesssim 2^{-Lk}\|f\|_p,$$
for some $L=L(p)>0$ independent of $k$.
\end{proposition}
\begin{proof}Fix $k$. We can assume $p$ is an integer and then use interpolation.
The value of $L$ will change throughout the argument.
For each $m\ge 0$, let $\F^m$ denote one of the forests $\F^{m,n}_k$ from Lemma \ref{l1}.
By the triangle inequality, it will suffice to prove that
$$\sum_{m}\|C_{\F^m}f\|_p^p+\sum_{m_1\le m_2\le\ldots\le m_p\atop{p^2+m_1<m_p}}\int|C_{\F^{m_1}}f|\prod_{i=2}^p|C_{\F^{m_i}}f|\lesssim 2^{-Lk}\|f\|_p^p.$$
Note that the first term is taken care of by Lemma \ref{cororthocmnheri6t783489ry958y}.
We next focus on the second term. The restriction $p^2+m_1<m_p$ in the summation can be achieved by splitting the integers in classes of residues modulo $p^2$ and using the triangle inequality. This separation will be used to achieve extra decay.

For each $T\in\F^{m_1}$ let as before $\J_T$  be the collection of all maximal dyadic intervals $J\subset I_T$ such that $J$ contains no $I_P$ with $P\in T$. Denote by $S_T$ the support $$\bigcup_{P\in T}(I_P\cap N^{-1}(\omega_P))$$ of $C_{T}f$.
Recall that the sets $S_T$ are pairwise disjoint.

 Since for each  $P\in T$ we have $I_P\subset E_k^{m_1-1}$, \eqref{e4}-\eqref{e6} will imply that
$$|I_P\cap (\operatorname{supp}C_{\F^{m_p}}f)|\le e^{-(m_p-m_1-1)k}|I_P|.$$
Since the dyadic parent of $J$ must equal $I_P$ for some $P\in T$, we conclude that also
\begin{equation}
\label{e9}
|J\cap (\operatorname{supp}C_{\F^{m_p}}f)|\le e^{-(m_p-m_1-1)k}|J|.
\end{equation}
Thus
$$\sum_{m_1\le m_2\le\ldots\le m_p\atop{m_1+p^2<m_p}}\int|C_{\F^{m_1}}f|\prod_{i=2}^p|C_{\F^{m_i}}f|\le \sum_{m_1\le m_2\le\ldots\le m_p\atop{m_1+p^2<m_p}}\sum_{T\in \F^{m_1}}\sum_{J\in\J_T}\int_J|C_{T}f|\prod_{i=2}^p|C_{\F^{m_i}}f|=$$
$$\sum_{m_1\le m_2\le\ldots\le m_p\atop{m_1+p^2<m_p}}\sum_{T\in \F^{m_1}}\sum_{J\in\J_T}\int_{J\cap (\operatorname{supp}C_{\F^{m_p}}f)}|C_{T}f|\prod_{i=2}^p|1_{S_T}C_{\F^{m_i}}f|$$

By invoking \eqref{jvgirtg782r]o3olfkjtiohj}, \eqref{e9} Lemma \ref{lemasingtree1} and H\"older's inequality this can further be bounded by
$$\sum_{m_1\le m_2\le\ldots\le m_p\atop{m_1+p^2<m_p}}\sum_{T\in \F^{m_1}}\sum_{J\in\J_T}|J\cap (\operatorname{supp}C_{\F^{m_p}}f)|^{1/p}\inf_{x\in J}M(V_Tf)(x)\prod_{i=2}^p(\int_J|1_{S_T}C_{\F^{m_i}}f|^p)^{1/p}\le$$
$$\sum_{m_1\le m_2\le\ldots\le m_p\atop{m_1+p^2<m_p}}e^{-(m_p-m_1-1)k/p}\sum_{T\in \F^{m_1}}\sum_{J\in\J_T}(\int_J[M(V_Tf)]^p)^{1/p}\prod_{i=2}^p(\int_J|1_{S_T}C_{\F^{m_i}}f|^p)^{1/p}\lesssim$$
$$\sum_{m_1\le m_2\le\ldots\le m_p\atop{m_1+p^2<m_p}}e^{-(m_p-m_1-1)k/p}\sum_{T\in \F^{m_1}}\|\Pi_Tf\|_p\prod_{i=2}^p(\int|1_{S_T}C_{\F^{m_i}}f|^p)^{1/p}\le$$
$$\sum_{m_1\le m_2\le\ldots\le m_p\atop{m_1+p^2<m_p}}e^{-(m_p-m_1-1)k/p}(\sum_{T\in\F^{m_1}}\|\Pi_{T}f\|_p^p)^{1/p}\prod_{i=2}^p\|C_{\F^{m_i}}f\|_p.$$
In the last line we have used the orthogonality of $\Pi_Tf$ and the disjointness of $S_T$.

Using \eqref{elpbencanz}, the forest estimate in Proposition \ref{forestesftttttimate}, H\"older's inequality  and the fact that
$$C_{\F^{m}}f=C_{\F^{m}}(\Pi_{\F^{m}}f),$$
we upgrade the last estimate to
$$\sum_{m_1\le m_2\le\ldots\le m_p\atop{m_1+p^2<m_p}}\int|C_{\F^{m_1}}f|\prod_{i=2}^p|C_{\F^{m_i}}f|\lesssim 2^{-\frac{k(p-1)}{p}}\sum_{m_1\le m_2\le\ldots\le m_p\atop{m_1+p^2<m_p}}e^{-(m_p-m_1-1)k/p}\prod_{i=1}^p\|\Pi_{\F^{m_i}}f\|_p$$$$\lesssim e^{-kp}\sum_m\|\Pi_{\F^m}f\|_p^p.$$
The Proposition will now follow from Corollary \ref{corortho}.
\end{proof}

\section{A proof without any appeal to the choice function}
\label{nochoice}

Let $\psi,\phi$ be Schwartz functions supported on $\T=[-1/2,1/2)$ such that $\int \psi=1$, $\int \phi=0$. Define the kernel
$$K(t,s)=\sum_{k<0}\psi_k(t)\phi_k(s)$$
where $\psi_k(t)=2^{-k}\psi(t2^{-k})$ and $\phi_k(s)=2^{-k}\phi(s2^{-k})$. Consider the trilinear form
$$\Lambda(F_1,F_2,F_3)=\int_{\T^4}F_1(x+t,y+s)F_2(x+s,y)F_3(x,y)K(t,s)dxdydtds,$$
where $+$ denotes the summation modulo one  on the torus. This is an example of a dualized two dimensional bilinear Hilbert transform.

Motivated in part by questions from Ergodic Theory, in \cite{DT} it is proved that
\begin{equation}
\label{fnreifodou4tiou4igtio}
|\Lambda(F_1,F_2,F_3)|\lesssim \|F_1\|_{p_1}\|F_2\|_{p_2}\|F_3\|_{p_3},
\end{equation}
whenever $1/{p_1}+1/{p_2}+1/{p_3}=1$ and $2<p_i<\infty$.

Interestingly, this implies the $L^p,p>2$ boundedness of the Carleson operator defined as
$$Cf(x)=\sup_{N\in\N}|\int_{\T}f(x+s)e^{i  Ns}\frac{ds}{s}|,$$
and in particular, the almost everywhere convergence of the Fourier series $S_nf(x)$ in the $p>2$ regime.
To see this, it suffices  to apply \eqref{fnreifodou4tiou4igtio} to $F_1(x,y)=f(y)$, $F_2(x,y)=e^{ixN(y)}g(y)$, $F_3(x,y)=e^{-ixN(y)}h(y)$ with $\|g\|_{p_2}=\|h\|_{p_3}=1$ and to an appropriate $\phi$ such that $\sum_{k<0}\phi_k(s)=1/s$ for $s\in [-1/4,1/4]\setminus\{0\}$.

We will prove the Walsh model analogue of \eqref{fnreifodou4tiou4igtio} for the dyadic kernel
$$K^W(t,s)=\sum_{k\ge 0}2^{\frac{3k}2}1_{I_k}(t)h_{I_k}(s),$$
where $I_k=[0,2^{-k}]$ and $h_{I_k}$ is the $L^2$ normalized Haar function. Define now
$$\Lambda^{W}(F_1,F_2,F_3)=\int_{[0,1]^4}F_1(x\oplus t,y\oplus s)F_2(x\oplus s,y)F_3(x,y)K^W(t,s)dxdydtds.$$
\begin{theorem}
\label{w,.qsklwqoid9348r94}
$$|\Lambda^{W}(F_1,F_2,F_3)|\lesssim \|F_1\|_{p_1}\|F_2\|_{p_2}\|F_3\|_{p_3},$$
whenever $1/{p_1}+1/{p_2}+1/{p_3}=1$ and $2<p_i<\infty$.

\end{theorem}

The definition of $\Lambda(F_1,F_2,F_3)$ involves no choice function $N(x)$, so we recover a proof of the boundedness of the Carleson operator  which makes no mention of it! The proof of Theorem \ref{w,.qsklwqoid9348r94} is very close in spirit to the proof of the boundedness of the Walsh model of the bilinear Hilbert transform \cite{BHT}. This shows once more the deep connection between these two operators.

Define the projection operators
$$\pi_{\omega}^1F(x,y)=\F_{W}^{-1}[1_\omega(\xi)\F_W(F)(\xi,\eta)](x,y)$$
$$\pi_{\omega}^2F(x,y)=\F_{W}^{-1}[1_\omega(\eta)\F_W(F)(\xi,\eta)](x,y).$$
It can be easily checked that
$$\Lambda^{W}(F_1,F_2,F_3)=\int_{[0,1]^2}\sum_{\omega\in\D_+:|\omega|\ge 2}\pi^1_{[0,\frac{|\omega|}{2}]}\pi_{\omega_l}^2F_1(x,y)\pi^1_{\omega_u}F_2(x,y)\pi^1_{\omega_u}F_3(x,y)dxdy+$$
$$\int_{[0,1]^2}\sum_{\omega\in\D_+:|\omega|\ge 2}\pi^1_{[0,\frac{|\omega|}{2}]}\pi_{\omega_u}^2F_1(x,y)\pi^1_{\omega_l}F_2(x,y)\pi^1_{\omega_l}F_3(x,y)dxdy,$$
where $\omega_l,\omega_u$ denote the left and right halves of $\omega$.

By symmetry, it will suffice to estimate the first integral, which we denote by $\Lambda^{W,1}(F_1,F_2,F_3)$.
Let $\Omega$ be a collection of intervals in $\D_+$.
 Define
$$\Lambda_\Omega(F_1,F_2,F_3)=\int_{[0,1]^2}\sum_{\omega\in\Omega}\pi^1_{[0,\frac{|\omega|}{2}]}\pi_{\omega_l}^2F_1(x,y)\pi^1_{\omega_u}F_2(x,y)\pi^1_{\omega_u}F_3(x,y)dxdy.$$
The first case of interest is when the intervals in $\Omega$ are nested. This will be a precursor of the tree estimate in Lemma \ref{efjreg458t6ghu56}.
\begin{proposition}
\label{propo1}
Let $\Omega$ be a collection of intervals in $\D_+$ which contain a point $\xi$. Then
$$|\Lambda_\Omega(F_1,F_2,F_3)|\lesssim \|F_1\|_{p_1}\|F_2\|_{p_2}\|F_3\|_{p_3},$$
whenever $1/{p_1}+1/{p_2}+1/{p_3}=1$ and $1<p_i<\infty$.
\end{proposition}
\begin{proof}
By splitting $\Omega$ in two parts, it suffices to assume that either $\xi\in\omega_l$ for all $\omega$ or $\xi\in\omega_u$ for all $\omega$. In the first case, note that the intervals $\omega_u$ are pairwise disjoint. We estimate using H\"older's inequality and the boundedness of the square function $S$ associated with the intervals $\omega_u$ on the first component
$$|\Lambda_\Omega(F_1,F_2,F_3)|\lesssim$$
 $$\int_{[0,1]^2}\sup_\omega|\pi^1_{[0,\frac{|\omega|}{2}]}\pi_{\omega_l}^2F_1(x,y)|(\sum_{\omega\in\Omega}|\pi^1_{\omega_u}F_2(x,y)|^2)^{1/2}(\sum_{\omega\in\Omega}|\pi^1_{\omega_u}F_3(x,y)|^2)^{1/2}dxdy\le$$
 $$\|M(F_1)\|_{p_1}\|SF_2\|_{p_2}\|SF_3\|_{p_3}\lesssim \|F_1\|_{p_1}\|F_2\|_{p_2}\|F_3\|_{p_3}.$$

In the second case note that $\omega_{l}$ are pairwise disjoint. We will run a standard telescoping argument. Denote by $\tilde{\Omega}$ the collection of all set differences between intervals $\omega_{u}$ of consecutive length. Note that each $\tilde{\omega}=\omega_u'\setminus\omega_u\in\tilde{\Omega}$ is the union of at most two intervals whose length is at least $|\omega_u|$.  We call $\omega^{ext}$ the complement (in $\R_+$) of the largest interval $\omega_u$, $\omega\in\Omega$, and $F_i^{ext}=\pi^1_{\omega^{ext}}F_i$ for $i=2,3$. Note that trivially, for each $\omega\in\Omega$

$$\pi^1_{\omega_u}F_i=F_i^{ext}-\sum_{\tilde{\omega}\in \tilde{\Omega}:\;\tilde{\omega}\not\subset \omega_u}\pi^1_{\tilde{\omega}} F_i,\;\;i=2,3.$$
We can now write
$$\Lambda_\Omega(F_1,F_2,F_3)=$$
\begin{equation}
\label{e24}
\int_{[0,1]^2}\sum_{\omega\in\Omega}\pi^1_{[0,\frac{|\omega|}{2}]}\pi_{\omega_l}^2F_1(x,y)F_2^{ext}(x,y)F_3^{ext}(x,y)dxdy-
\end{equation}
\begin{equation}
\label{e25}
-\int_{[0,1]^2}\sum_{\omega\in\Omega}\sum_{\tilde{\omega}\in\tilde{\Omega}:\;\tilde{\omega}\not\subset \omega_u}\pi^1_{[0,\frac{|\omega|}{2}]}\pi_{\omega_l}^2F_1(x,y)F_2^{ext}(x,y)\pi^1_{\tilde{\omega}}F_3(x,y)dxdy-
\end{equation}
\begin{equation}
\label{e26}
-\int_{[0,1]^2}\sum_{\omega\in\Omega}\sum_{\tilde{\omega}\in\tilde{\Omega}:\;\tilde{\omega}\not\subset \omega_u}\pi^1_{[0,\frac{|\omega|}{2}]}\pi_{\omega_l}^2F_1(x,y)\pi^1_{\tilde{\omega}}F_2(x,y)F_3^{ext}(x,y)dxdy+
\end{equation}
\begin{equation}
\label{e27}
+\int_{[0,1]^2}\sum_{\omega\in\Omega}\sum_{\tilde{\omega}\in\tilde{\Omega}:\;\tilde{\omega}\not\subset \omega_u}\sum_{\tilde{\omega'}\in\tilde{\Omega}:\;\tilde{\omega'}\not\subset \omega_u}\pi^1_{[0,\frac{|\omega|}{2}]}\pi_{\omega_l}^2F_1(x,y)\pi^1_{\tilde{\omega}}F_2(x,y)\pi^1_{\tilde{\omega'}}F_3(x,y)dxdy.
\end{equation}
The term \eqref{e24} is controlled by H\"older's inequality and the boundedness of the two dimensional singular integral operator $TF=\sum_{\omega\in\Omega}\pi^1_{[0,\frac{|\omega|}{2}]}\pi_{\omega_l}^2F$. To control  \eqref{e27} we note that the term corresponding to a triple $\omega,\tilde{\omega},\tilde{\omega'}$ is nonzero only when $\tilde{\omega},\tilde{\omega'}$
 are adjacent to each other. This is simple computation involving the Walsh transform that will diagonalize the sum. Then use Cauchy-Schwartz and the boundedness of the square function $(\sum_{\tilde{\omega}\in\tilde{\Omega}}|\pi^1_{\tilde{\omega}}F|^2)^{1/2}$ and of the maximal truncation operator $$T^*F=\sup_{\delta>0}|\sum_{\omega\in\Omega:\;|\omega|>\delta}\pi^1_{[0,\frac{|\omega|}{2}]}\pi_{\omega_l}^2F|.$$

 A similar argument will take care of the terms \eqref{e25} and \eqref{e26}, the details are left to the reader.

\end{proof}

The next step is to discretize the form $\Lambda$ in both time and frequency.

\begin{definition}
A one and a half dimensional bitile (tile) $P=R_P\times \omega_P$, for short $\frac32$ - bitile (tile), is a product of  a dyadic square $R_P=I_P\times J_P\in \D_+^2$ and a dyadic interval $\omega_P\in\D_+$ with $|\omega_P|=2|I_P|^{-1}$ $(|\omega_P|=|I_P|^{-1})$. If $P$ is a $\frac32$ - bitile, we define as before the upper and lower $\frac32$ - tiles $P_u=R_P\times \omega_{P_u}$, $P_l=R_P\times \omega_{P_l}$.
\end{definition}

Call $\P_{all}^{3/2}$ the entire collection of $\frac32$ - bitiles. There is a natural partial order on $\P_{all}^{3/2}$ which, by abusing earlier notation will  be denoted by $\le$. The reader can correctly anticipate that $P\le P'$ will stand for $R_P\subset R_{P'}$ and $\omega_{P'}\subset\omega_P$. Convexity on $\P_{all}^{3/2}$ will be understood with respect to this order.
\begin{definition}
Let $R_T\subset \D_+^2$ be a dyadic square and let $\xi_T\in\R_+$. A $\frac32$ - tree $T$ with top data $(R_T,\xi_T)$ is a collection of $\frac32$ -  bitiles in $\P_{all}^{3/2}$ such that $R_P\subset R_T$ and $\xi_T\in\omega_P$ for each $P\in T$.
\end{definition}

With each $\frac32$ - bitile $P$ we will identify two regions in the phase space $\R_+^4=\R_+^2\times \widehat{\R_+^2}$. One is the region covered by the two dimensional bitile $R_P\times [0,|\omega_P|]\times \omega_P$. The other one is $R_P\times \omega_P \times \R_+ $, which is the union of two dimensional bitiles of the form $R_P\times \omega_P \times \omega$. We will denote by $\mathcal C_P$ the collection of all these two dimensional bitiles. It is easy to check that if $\P\subset \P_{all}^{3/2}$ is convex then both collections of bitiles $\{R_P\times [0,|\omega_P|]\times \omega_P:P\in\P\}$ and
$\cup_{P\in\P}\mathcal C_P$ are convex with respect to the two dimensional order. We will denote by $\Pi^2_\P F$ and $\Pi^{3/2}_\P F$ the phase space projections associated with the two collections (cf. Remark \ref{rem1})

Note that $\Pi^{3/2}_\P F$ can be thought of as being a one and a half dimensional projection, since this operator produces no localization on the second frequency component. For example, if $p$ is a $\frac32$ - tile then
$$\Pi_{p}^{3/2}F(x,y)=\left[\int F(x',y)W_{I_p\times\omega_p}(x')dx'\right]W_{I_p\times\omega_p}(x)1_{J_{p}}(y)=$$
$$=\sum_{|\omega|=|\omega_p|}\langle F,W_{I_p\times\omega_p}W_{J_p\times\omega}\rangle W_{I_p\times\omega_p}(x)W_{J_p\times\omega}(y).$$

For a convex $\frac32$ - tree $T$ define
$$\Lambda_T(F_1,F_2,F_3)=\int_{[0,1]^2}\sum_{P\in T}1_{R_P}(x,y)\pi^1_{[0,\frac{|\omega_P|}{2}]}\pi_{\omega_{P_l}}^2F_1(x,y)\pi^1_{\omega_{P_u}}F_2(x,y)\pi^1_{\omega_{P_u}}F_3(x,y)dxdy,$$
and observe that variants of \eqref{e16} will imply that
$$\Lambda_T(F_1,F_2,F_3)=\Lambda_T(\Pi_T^2(F_1),\Pi_T^{3/2}(F_2),\Pi_T^{3/2}(F_3))$$
An argument very similar to the one in Proposition \ref{propo1} will prove
$$
|\Lambda_T(F_1,F_2,F_3)|\lesssim \|F_1\|_{s_1}\|F_2\|_{s_2}\|F_3\|_{s_3},$$
whenever $1/{s_1}+1/{s_2}+1/{s_3}=1$ and $1<s_i<\infty$. Combining these we get
\begin{equation}
\label{e28}
|\Lambda_T(F_1,F_2,F_3)|\lesssim \|\Pi_T^2(F_1)\|_{s_1}\|\Pi_T^{3/2}(F_2)\|_{s_2}\|\Pi_T^{3/2}(F_3)\|_{s_3}
\end{equation}

\begin{definition}Let $\P\subset\P^{3/2}_{all}$ be a collection of $\frac32$ - bitiles. For $F:\R_+^2\to \C$ define
$$\size_{F,2}(\P)=\sup_{P\in\P}\frac{\|\Pi_P^2F\|_2}{|R_P|^{1/2}}$$
$$\size_{F,3/2}(\P)=\sup_{P\in\P}\frac{\|\Pi_P^{3/2}F\|_2}{|R_P|^{1/2}}.$$
\end{definition}
\begin{lemma}[Tree estimate]
\label{efjreg458t6ghu56}
Let $T\subset\P_{all}^{3/2}$ be a convex $\frac32$ - tree. Then
$$\|\Pi_T^2F\|_{p}\lesssim |R_T|^{1/p}\size_{F,2}(T),\text{  for each }1\le p\le\infty.$$
If in addition $\|F\|_{\infty}\le 1$ then
$$\|\Pi_T^{3/2}F\|_{p}\lesssim |R_T|^{1/p}(\size_{F,{3/2}}(T))^{2/p},\text{  for each }2\le p\le \infty.$$\end{lemma}
\begin{proof}
The crucial observation in both cases is that
\begin{equation}
\label{e30}
\Pi_T^{i}F(x,y)=\Pi_P^{i}F(x,y),
\end{equation}
for each $i\in\{\frac32,2\}$, where $P$ is minimal $\frac32$ - bitile in $T$ such that $(x,y)\in R_P$. This follows like in the proof of Proposition \ref{hfuyhrufywiei3u884u8}. The first inequality then follows by noting that $\|\Pi_P^{2}F\|_{\infty}\lesssim \|\Pi_P^{2}F\|_{2}$, like in the one dimensional case.

Let us now analyze the second inequality. Let $\P_T$ be the collection of all $\frac32$ - bitiles in $T$ such that for each $P\in \P_T$ there is $(x,y)$ such that $P$ is the minimal $\frac32$ - bitile  in $T$ with $(x,y)\in R_P$. It is easy to see that
\begin{equation}
\label{e31}
\sum_{P\in \P_T}|R_P|\lesssim |R_T|.
\end{equation}
Indeed, if $P\in\P_T$, by the convexity of $T$ at least one of the four dyadic children of $R_P$ will contain no $R_{P'}$ with $P'\in\P_T$.

Next we observe that for each $\frac32$ - bitile $P$,
\begin{equation}
\label{e29}
\|\Pi_P^{3/2}F\|_{\infty}\le \|F\|_{\infty},
\end{equation}
which together with H\"older's inequality gives, if $\|F\|_{\infty}\le 1$ and $2\le p\le \infty$
$$\|\Pi_P^{3/2}F\|_{p}\le |R_P|^{1/p}(\size_{F,3/2}(P))^{2/p}.$$
The result now follows by combining this with \eqref{e30} and \eqref{e31}.
\end{proof}

The proof of Proposition \ref{2erbyu76i87o} applies with essentially  no  modification to prove the following variant

\begin{proposition}[Size decomposition]
\label{2erbyu76i87onewdd}

Let $i$ be either $2$ or $3/2$. Let $F:\R_+^2\to \C$.
Let $\P$ be a finite convex collection of $\frac32$ - bitiles.

Then
$$
\P= \bigcup_{2^{-n}\le \size_{F,i}(\P)}\P_{n} \cup \P_{null}
,$$
such that \begin {itemize}
\item[$\cdot$]$\size_{F,i}(\P_{n} )\leq 2^{-n}$:
\item[$\cdot$] $\P_{n}$ is a convex forest with convex $\frac32$ - trees $T\in\F_n$ satisfying
$$
\sum_{T\in\F_n}  |R_T|\lesssim 2^{2n}\|F\|^2_2,
$$

\item [$\cdot$]$\Pi_P^{i}F\equiv 0$ for each $P\in\P_{null}$
\end{itemize}\end{proposition}
\begin{proof}[of Theorem \ref{w,.qsklwqoid9348r94}]
By multilinear restricted type interpolation it suffices to prove
$$|\Lambda^{W,1}(F_1,F_2,F_3)|\lesssim |E_1|^{1/p_1}|E_2|^{1/p_2}|E_3|^{1/p_3}$$ whenever
 $|F_i|\le 1_{E_i}$ for finite measure subsets $E_i$ of $\R_+^2$. Fix $2<p_i<\infty$ for the rest of the proof. We begin by observing that
$$\Lambda^{W,1}(F_1,F_2,F_3)=\sum_{P\in\P_{all}^{3/2}}\Lambda_{P}(F_1,F_2,F_3),$$
and thus it further suffices to prove
$$|\sum_{P\in\P}\Lambda_{P}(F_1,F_2,F_3)|\lesssim |E_1|^{1/p_1}|E_2|^{1/p_2}|E_3|^{1/p_3},$$
for all convex, finite $\P\subset \P_{all}^{3/2}$.

Note that by \eqref{e29} and the natural two dimensional extension of  Proposition \ref{djhfdftyert}, we have
$$\size_{F_1,2}(\P), \size_{F_2,3/2}(\P), \size_{F_3,3/2}(\P)\lesssim 1.$$
Let $\P_n^{(1)}$, $\P_n^{(2)}$, $\P_n^{(3)}$ be the collections provided by Proposition \ref{2erbyu76i87onewdd}, relative to $\size_{F_1,2}(\P)$, $\size_{F_2,3/2}(\P)$ and  $\size_{F_3,3/2}(\P)$, respectively. Define $\P_{n_1,n_2,n_3}=\P_{n_1}^{(1)}\cap \P_{n_2}^{(2)}\cap \P_{n_3}^{(3)}$. Note that
\begin{equation}
\label{d,lmcerugop34r956pl}
\sum_{P\in\P}\Lambda_{P}(F_1,F_2,F_3)=\sum_{2^{-n_1},2^{-n_2},2^{-n_3}\lesssim 1}\sum_{P\in \P_{n_1,n_2,n_3}}\Lambda_{P}(F_1,F_2,F_3).
\end{equation}
Reasoning like in the proof of Theorem \ref{thm:main} from Section \ref{sec:Lac-Thi}, we organize $\P_{n_1,n_2,n_3}$ as a forest in three different ways, with the $L^1$ norm of the counting function of the tops bounded by $2^{2n_1}|E_1|$, $2^{2n_2}|E_2|$ and $2^{2n_3}|E_3|$, respectively.

Pick now $1>\alpha>\max\{\frac2{p_2},\frac2{p_3}\}$. Using \eqref{d,lmcerugop34r956pl}, \eqref{e28} with $s_2=s_3=2/\alpha$ and then Lemma \ref{efjreg458t6ghu56}  we get
$$|\sum_{P\in\P}\Lambda_{P}(F_1,F_2,F_3)|=\sum_{2^{-n_1},2^{-n_2},2^{-n_3}\lesssim 1}2^{-n_1-n_2\alpha-n_3\alpha}\min(2^{2n_1}|E_1|, 2^{2n_2}|E_2|, 2^{2n_3}|E_3|)\le$$
$$\sum_{2^{-n_1},2^{-n_2},2^{-n_3}\lesssim 1}2^{-n_1-n_2\alpha-n_3\alpha}(2^{2n_1}|E_1|)^{1/p_1} (2^{2n_2}|E_2|)^{1/p_2}(2^{2n_3}|E_3|)^{1/p_3}$$$$\lesssim |E_1|^{1/p_1}|E_2|^{1/p_2}|E_3|^{1/p_3}.$$
 The argument is now complete.
\end{proof}

\end{document}